\documentclass[11pt,a4paper]{article}
\usepackage{a4wide}
\usepackage{amsmath,amssymb}
\usepackage[english]{babel}
\usepackage{graphicx}
\usepackage[numbers]{natbib}
\usepackage{theorem}
\usepackage{hyperref}
\hypersetup{
    bookmarksopen=false,
    bookmarksnumbered=true,
    pdftitle={A Polling Model with Multiple Priority Levels},
    pdfauthor={M.A.A. Boon, I.J.B.F. Adan, O.J. Boxma}
}
\ifpdf
  \hypersetup{colorlinks=true,linkcolor=black,urlcolor=black,citecolor=black,pdfpagemode=UseOutlines,plainpages=false,pdfpagelabels}
\else
  \hypersetup{colorlinks=false}
\fi

\theorembodyfont{\upshape}
\theoremheaderfont{\bfseries}

\newtheorem{theorem}{Theorem}[section]
\newtheorem{property}[theorem]{Property}
\newtheorem{lemma}[theorem]{Lemma}
\newtheorem{remark}[theorem]{Remark}

\newenvironment{proof}{\par\textbf{Proof:}\\}{\hfill$\square$\par}
\numberwithin{equation}{section}

\newcommand{\lz}{\sum_{j=1}^N\lambda_j(1-z_j)}
\newcommand{\ee}{\textrm{e}}
\newcommand{\dd}{\,\textrm{d}}
\providecommand{\href}[2]{#2}

\begin{document}
\title{A Polling Model with Multiple Priority Levels\footnote{The research was done in the framework of the BSIK/BRICKS project, and of the European Network of Excellence Euro-FGI.}}
\author{M.A.A. Boon\footnote{\textsc{Eurandom} and Department of Mathematics and Computer Science, Eindhoven University of Technology, P.O. Box 513, 5600MB Eindhoven, The Netherlands}\\\href{mailto:marko@win.tue.nl}{marko@win.tue.nl} \and I.J.B.F. Adan\footnotemark[2]\\\href{mailto:iadan@win.tue.nl}{iadan@win.tue.nl} \and O.J. Boxma\footnotemark[2]\\\href{mailto:boxma@win.tue.nl}{boxma@win.tue.nl}}
\date{August, 2009}
\maketitle

\begin{abstract}
In this paper we consider a single-server cyclic polling system. Between visits to successive queues, the server is delayed by a random switch-over time. The order in which customers are served in each queue is determined by a priority level that is assigned to each customer at his arrival. For this situation the following service disciplines are considered: gated, exhaustive, and globally gated. We study the cycle time distribution, the waiting times for each customer type, the joint queue length distribution of all priority classes at all queues at polling epochs, and the steady-state marginal queue length distributions for each customer type.

\end{abstract}

\section{Introduction}\label{intro}

A polling model is a single-server system in which the server visits $N$ queues $Q_1, \dots, Q_N$ in cyclic order. Customers that arrive at $Q_i$ are referred to as type $i$ customers.
The special feature of the model considered in the present paper is that, within a customer type, we distinguish multiple priority levels. More specifically, we study a polling system which consists of $N$ queues, $Q_1, \dots, Q_N$, and $K_i$ priority levels in $Q_i$. The exhaustive, gated and globally gated service disciplines are studied.

Our motivation to study a polling model with priorities is that scheduling through the introduction of priorities in a polling system can improve the performance of the system significantly without having to purchase additional resources \cite{wierman07}.
Priority polling systems can be used to study the Bluetooth and IEEE 802.11 wireless LAN protocols, or scheduling policies at routers and I/O subsystems in web servers. For example, Quality-of-Service (QoS) has become a very important issue in wireless LAN protocols because delays in responses or data transmissions do not have a major impact on, e.g., web browsing or email traffic, but video transmissions and Voice over Wireless LAN (VoWLAN) are very sensitive to delays or loss of data. The 802.11e amendment introduces different priority levels to differentiate between types of data in order to improve QoS for streaming traffic. In production environments the introduction of priorities can lead to smaller mean waiting times (see, e.g., \cite{winandsPhD}). It is shown in \cite{boonadanboxma2queues2008,wierman07} that assigning highest priority to jobs with a service requirement below a certain threshold level may reduce overall mean waiting times.
Priority polling models also can be used to study traffic intersections where conflicting traffic flows face a green light simultaneously; e.g. traffic which takes a left turn may have to give right of way to conflicting traffic that moves straight on, even if the traffic light is green for both traffic flows. The last application area of polling models with priorities that we mention here, is health care (see, e.g., \cite{cicin2001}).

Although there is an extensive amount of literature available on polling systems (see, e.g., the surveys of Takagi \cite{takagi1988qap}, Levy and Sidi \cite{levysidi90}, and Vishnevskii and Semenova \cite{vishnevskiisemenova06}), very few papers treat priorities in polling models. Most of these papers only provide approximations or focus on pseudo-conservation laws \cite{fournierrosberg91,shimogawatakahashi88}. Wierman, Winands and Boxma \cite{wierman07} have obtained exact mean waiting time results using the Mean Value Analysis (MVA) framework for polling systems, developed in \cite{winands06}. The MVA framework can only be used to find the first moment of the waiting time distribution for each customer type, and the mean residual cycle time. In a recent report \cite{boonadanboxma2queues2008} we have studied a polling model with two queues that are served according to the exhaustive, gated or globally gated service discipline. The first of these two queues contains customers of two priority classes. The main contribution of \cite{boonadanboxma2queues2008} is the derivation of Laplace-Stieltjes Transforms (LSTs) of the distributions of the marginal waiting times for each customer type; in particular it turned out to be possible to obtain exact expressions for the waiting time distributions of both high and low priority customers at a queue of a polling system. Probability Generating Functions (GFs) have been derived for the joint queue length distribution at polling epochs, and for the steady-state marginal queue length distribution of the number of customers at an arbitrary epoch. In the present paper these results are generalised to a polling model with $N$ queues and $K_i$ priority levels in $Q_i$ $(i=1,\dots,N)$. The actual probability distributions can be obtained by numerical inversion of the LSTs and GFs. A very efficient technique for numerical inversion of GFs and LSTs in polling models is discussed in \cite{choudhurywhitt96}.

The present paper is structured as follows: Section \ref{general} gathers known results of nonpriority polling models which we shall need in the present study. Sections \ref{priorities} (gated and exhaustive), and \ref{globallygated} (globally gated) give new results on the priority polling model. In each of the sections we successively discuss the joint queue length distribution at polling epochs, the cycle time distribution, the marginal queue length distributions and waiting time distributions. The mean waiting times are given at the end of each section.

\section{The nonpriority polling model}\label{general}

The model that is considered in this section, is a polling model with $N$ queues ($Q_1, \dots, Q_N$) without priorities.
Each queue is served according to either the gated, or the exhaustive service discipline. The gated service discipline states that during a visit to $Q_i$, only those type $i$ customers are served who are present at the moment that the server arrives at $Q_i$. All type $i$ customers that arrive during the visit to $Q_i$ will be served during the next cycle. A cycle is the time between two successive visit beginnings (or completions) at a queue. The exhaustive service discipline states that when the server arrives at $Q_i$, all type $i$ customers that are present at that polling epoch, and all type $i$ customers that arrive during this particular visit to $Q_i$, are served until no type $i$ customer is present in the system. In Section \ref{globallygated} we also consider the globally gated service discipline, which is similar to the gated service discipline, except for the fact that the symbolic gate is being set at the beginning of a cycle for \emph{all} queues. This means that during a cycle only those customers will be served that were present at the beginning of that cycle.

Customers of type $i$ arrive at $Q_i$ according to a Poisson process with arrival rate $\lambda_i$. Service times can follow any distribution. The LST of the distribution of the generic service time $B_i$ of type $i$ customers is denoted by $\beta_i(\cdot)$.
The fraction of time that the server is serving customers of type $i$ equals $\rho_i := \lambda_i E(B_i)$. Switches of the server from $Q_i$ to $Q_{i+1}$ (all indices modulo $N$), require a switch-over time $S_i$. The LST of this switch-over time distribution is denoted by $\sigma_i(\cdot)$. The fraction of time that the server is working (i.e., not switching) is $\rho := \sum_{i=1}^N \rho_i$. We assume that $\rho < 1$, which is a necessary and sufficient condition for the steady state distributions of cycle times, queue lengths and waiting times to exist. We assume that all the usual independence assumptions apply to the model under consideration.

This model has been extensively investigated. Tak\'acs \cite{takacs68} studied this model, but with only two queues, without switch-over times and only with the exhaustive service discipline. Cooper and Murray \cite{coopermurray69} analysed this polling system for any number of queues, and for both gated and exhaustive service disciplines. Eisenberg \cite{eisenberg72} obtained results for a polling system with switch-over times (but only exhaustive service) by relating the GFs of the joint queue length distributions at visit beginnings, visit endings, service beginnings and service endings. Resing \cite{resing93} and Fuhrmann \cite{fuhrmann81} both pointed out the relation between polling systems and Multitype Branching Processes (MTBPs). Resing \cite{resing93} considers the use of MTBPs with immigration in each state, which makes it possible to relate polling systems with and without switch-over times. His results can be applied to polling models in which each queue satisfies the following property:

\begin{property}\label{resingproperty}
If the server arrives at $Q_i$ to find $k_i$ customers there, then during the course of the server's visit, each of these $k_i$ customers will effectively be replaced in an i.i.d. manner by a random population having probability generating function $h_i(z_1,\dots,z_N)$, which can be any $N$-dimensional probability generating function.
\end{property}

We use this property, and the relation to Multitype Branching Processes, to find results for our polling system with multiple priority levels, and gated, globally gated, and exhaustive service discipline. Notice that, unlike the gated and exhaustive service disciplines, the globally gated service discipline does not satisfy Property \ref{resingproperty}. But the results obtained by Resing also hold for a more general class of polling systems, namely those which satisfy the following (weaker) property that is formulated in \cite{semphd}:

\begin{property}\label{borstproperty}
If there are $k_i$ customers present at $Q_i$ at the beginning (or the end) of a visit to $Q_{\pi(i)}$, with $\pi(i) \in \{1, \dots, N\}$, then during the course of the visit to $Q_i$, each of these $k_i$ customers will effectively be replaced in an i.i.d. manner by a random population having probability generating function $h_i(z_1,\dots,z_N)$, which can be any $N$-dimensional probability generating function.
\end{property}

Globally gated and gated are special cases of the  synchronised gated service discipline, which states that only customers in $Q_i$ will be served that were present at the moment that the server reaches the ``parent queue'' of $Q_i$: $Q_{\pi(i)}$. For gated service, $\pi(i) = i$, for globally gated service, $\pi(i) = 1$. The synchronised gated service discipline is discussed in \cite{khamisy92}, but no observation is made that this discipline is a member of the class of polling systems satisfying Property \ref{borstproperty} which means that results as obtained in \cite{resing93} can be extended to this model. 

Borst and Boxma \cite{borst97} combined the results of Resing \cite{resing93} and Eisenberg \cite{eisenberg72} to find a relation between the GFs of the marginal queue length distribution for polling systems with and without switch-over times, expressed in the Fuhrmann-Cooper queue length decomposition form \cite{fuhrmanncooper85}.

\subsection{Joint queue length distribution at polling epochs}\label{generaljoint}

As discussed in \cite{resing93}, polling systems which satisfy Property \ref{resingproperty} can be modelled as a MTBP. If the polling model has switch-over times, as is assumed in the present paper, we are dealing with a MTBP with immigration in each state. The immigration is formed by the customers that arrive during the switch-over times and all of their descendants. Polling models without switch-over times can be modelled as a MTBP with immigration in state zero (when the server is idle) only. Although the present paper studies a polling  model with switch-over times, it is possible to relate results from a model with switch-over times to a model without switch-over times, as is shown in \cite{borst97}.

For now, we choose the beginning of a visit to $Q_1$ as start of a cycle. Property \ref{resingproperty} states that each customer present at the beginning of a cycle will be ``replaced'' by customers of type $1, \dots, N$ during that cycle according to the probability generating function $h_i(z_1,\dots,z_N)$, which depends on the service discipline. If $Q_i$ $(i=1, \dots, N)$ receives gated service, a customer of type $i$ present at the beginning of the cycle is replaced by the type $1, \dots, N$ customers that arrive during his service: $h_i(z_1,\dots,z_N) = \beta_i(\lz)$. For exhaustive service, a type $i$ customer will not just be replaced by all type $1, \dots, N$ customers that arrive during his service, but also all customers that arrive during the service of all type $i$ customers that have arrived during the service of this tagged customer, etc., until no type $i$ customer is present in the system. The consequence is that this type $i$ customer has been replaced by the type $j$ customers $(j = 1, \dots, N; j \neq i)$ that have arrived during the busy period (BP) of type $i$ customers that is initiated by the particular type $i$ customer. This results in the expression $h_i(z_1,\dots,z_N) = \pi_i(\sum_{j\neq i} \lambda_j(1-z_j))$, where $\pi_i(\cdot)$ is the LST of a BP distribution in an $M/G/1$ system with only type $i$ customers, so it is the root in $(0,1]$ of the equation $\pi_i(\omega) = \beta_i(\omega + \lambda_i(1 - \pi_i(\omega)))$, $\omega \geq 0$ (cf. \cite{cohen82}, p. 250).

In order to find the joint queue length distribution at the beginning of a cycle, we have to define the immigration GF and the offspring GF analogous to \cite{resing93}. The first generation of offspring consists of the customers that have effectively replaced the customers present at the beginning of the cycle.
The offspring GFs for queues $N, N-1, \dots, 1$ are given below.
\begin{align*}
f^{(N)}(z_1, \dots, z_N) &= h_N(z_1, \dots, z_{N}),\\
f^{(i)}(z_1, \dots, z_N) &= h_i(z_1, \dots, z_{i},f^{(i+1)}(z_1,\dots,z_N),\dots,f^{(N)}(z_1,\dots,z_N)), \qquad i=1,\dots,N-1.\\
\end{align*}
We define the GF for the $n^\textrm{th}$ generation of offspring recursively:
\begin{align*}
f_n(z_1, \dots, z_N) &= (f^{(1)}(f_{n-1}(z_1, \dots, z_N)), \dots, f^{(N)}(f_{n-1}(z_1, \dots, z_N))), \\
f_0(z_1, \dots, z_N) &= (z_1, \dots, z_N).
\end{align*}
The immigration GF is the GF of the joint distribution of the numbers of customers that are present at the beginning of a cycle due to the immigration during the switch-over times of the previous cycle. These customers can be divided into two groups: the customers that arrived during a switch \emph{after} the service of their queue, and offspring of customers that arrived during a switch \emph{before} the service of their queue. The immigration GFs are:
\begin{align*}
g^{(N)}(z_1, \dots, z_N) &= \sigma_N(\lz), \\
g^{(i)}(z_1, \dots, z_N) &= \sigma_{i}(\sum_{j=1}^{i}\lambda_j(1-z_j)+\sum_{j=i+1}^N\lambda_j(1-f^{(j)}(z_1, \dots, z_N))), \qquad i=1,\dots,N-1.\\
\end{align*}
The total immigration GF is the product of these GFs:
\[
g(z_1, \dots, z_N) = \prod_{i=1}^N g^{(i)}(z_1, \dots, z_N).
\]
This leads to the following recursive expression for the GF of the steady-state joint queue length at the beginning of a cycle (starting with a visit to $Q_1$) :
\begin{equation}
P_1(z_1, \dots, z_N) = P_1(f_1(z_1, \dots, z_N))g(z_1, \dots, z_N).\label{vbirecursive}
\end{equation}
Resing \cite{resing93} shows how \eqref{vbirecursive} can be used to obtain moments from the marginal queue length distribution at the beginning of a cycle. It is shown in \cite{quine70} that iteration of \eqref{vbirecursive} leads to
\begin{equation}
P_1(z_1, \dots, z_N) = \prod_{n=0}^\infty g(f_n(z_1, \dots, z_N)).\label{p1}
\end{equation}
Resing \cite{resing93} proves that this infinite product converges if and only if $\rho < 1$.

We can relate the steady-state joint queue length distribution at other visit beginnings and endings to $P_1(z_1, \dots, z_N)$. We denote the GF of the joint queue length distribution at a visit beginning to $Q_i$ by $V_{b_i}(\cdot)$, so $V_{b_1}(\cdot) = P_1(\cdot)$. The queue length GF at a visit \emph{completion} to $Q_i$ is denoted by $V_{c_i}(\cdot)$. The following relation holds:
\begin{align}
V_{b_i}(z_1, \dots, z_N) &= V_{c_{i-1}}(z_1, \dots, z_N)\sigma_{i-1}(\lz) \nonumber\\
&= V_{b_{i-1}}(z_1,\dots,z_{i-2},h_{i-1}(z_1,\dots,z_N),z_{i},\dots,z_N)\sigma_{i-1}(\lz).\label{vbi}
\end{align}
Since $V_{b_1}(\cdot)$ is known, applying \eqref{vbi} $i-1$ times expresses $V_{b_i}(\cdot)$ into $V_{b_1}(\cdot) = P_1(\cdot)$.
Applying \eqref{vbi} $N$ times gives expression \eqref{vbirecursive}.

\subsection{Cycle time}\label{generalcycletime}

The cycle time, starting at a visit \emph{beginning} to $Q_1$, is the sum of the visit times to $Q_1, \dots, Q_N$, and the switch-over times $S_1, \dots, S_N$ which are independent of the visit times. Let $\theta_i(\cdot)$ denote the LST of the distribution of the time that the server spends at $Q_i$ due to the presence of one type $i$ customer there. For gated service $\theta_i(\cdot) = \beta_i(\cdot)$, for exhaustive service $\theta_i(\cdot) = \pi_i(\cdot)$. Furthermore, we define
\begin{align*}
\psi_i(\omega) &= \omega + \lambda_i(1 - \theta_i(\omega)), &&i=1,\dots,N,\\
\psi_{i,N}(\omega) &= \psi_{i+1}(\psi_{i+2}(\dots(\psi_N(\omega)))),&& i=1,\dots,N-1,\\
\psi_{N,N}(\omega) &= \omega.&&
\end{align*}
It is shown in \cite{boxmafralixbruin08} that the LST of the distribution of the cycle time $C_1$, $\gamma_1(\cdot)$, is related to $P_1(\cdot)$ as follows:
\begin{equation*}
\gamma_1(\omega) = \prod_{i=1}^N \sigma_i(\psi_{i,N}(\omega))\,P_1(\theta_1(\psi_{1,N}(\omega)), \dots, \theta_N(\psi_{N,N}(\omega))).
\end{equation*}
In the present paper we need expressions for the LST of the distribution of the cycle time starting at the beginning of a visit to an arbitrary queue. We denote the cycle time starting with a visit beginning to  $Q_j$ by $C_j$, and its LST by $\gamma_j(\cdot)$. It is straightforward to see that
\begin{equation}
\gamma_j(\omega) = \prod_{i=1}^N \sigma_i(\psi_{i,j-1}(\omega))\,V_{b_j}(\theta_1(\psi_{1,j-1}(\omega)), \dots, \theta_N(\psi_{N,j-1}(\omega))),\label{cycletimelst}
\end{equation}
where we use the following notation:
\begin{align*}
\psi_{i,j}(\omega) &= \psi_{i+1}(\dots(\psi_{j}(\omega))), &&j=1,\dots,N; i < j,\\
\psi_{i,j}(\omega) &= \psi_{i+1}(\dots(\psi_N(\psi_1(\dots(\psi_{j}(\omega)))))),&& j=1,\dots,N; i > j,\\
\psi_{j,j}(\omega) &= \omega,&& j=1,\dots,N.&&
\end{align*}
When analysing a queue with exhaustive service, it is convenient to define $C^*_j$ to be the time between two successive visit \emph{completions} to $Q_j$. The LST of its distribution, denoted by $\gamma^*_j(\cdot)$, is:
\begin{equation}
\gamma^*_j(\omega) = \prod_{i=1}^N \sigma_i(\psi^*_{i,j}(\omega))\,V_{c_j}(\theta_1(\psi_{1,j}(\omega)), \dots, \theta_N(\psi_{N,j}(\omega))),\label{cycletimlstVc}
\end{equation}
with $V_{c_j}(z_1,\dots,z_N) = V_{b_j}(z_1, \dots, z_{j-1}, h_j(z_1,\dots,z_N), z_{j+1}, \dots, z_N)$, and
\begin{align*}
\psi^*_{i,j}(\omega) &= \psi_{i,j}(\omega), && j=1,\dots,N; i \neq j,\\
\psi^*_{j,j}(\omega) &= \psi_{j+1}(\dots(\psi_N(\psi_1(\dots(\psi_{j}(\omega)))))),&&j=1,\dots,N.
\end{align*}

\subsection{Marginal queue lengths and waiting times}

We denote the GF of the steady-state marginal queue length distribution of $Q_i$ at the visit beginning by $\widetilde{V}_{b_i}(z) = V_{b_i}(1,\dots,1,z,1,\dots, 1)$, with $z$ as $i^\textrm{th}$ argument. Analogously we define $\widetilde{V}_{c_i}(\cdot)$.
It is shown in \cite{borst97} that the steady-state marginal queue length of $Q_i$ can be decomposed into two parts: the queue length of the corresponding $M/G/1$ queue with only type $i$ customers, and the queue length at an arbitrary epoch during the intervisit period of $Q_i$, denoted by ${L_{i|I}}$. An intervisit period of $Q_i$ is the time between a visit ending at $Q_i$, and the next visit beginning at $Q_i$. Borst \cite{semphd} shows that by virtue of PASTA, ${L_{i|I}}$ has the same distribution as the number of type $i$ customers seen by an arbitrary type $i$ customer arriving during an intervisit period, which equals
\begin{equation}
E(z^{L_{i|I}}) = \frac{E(z^{L_{i|I_{\textit{begin}}}}) - E(z^{L_{i|I_{\textit{end}}}})}{(1-z)(E(L_{i|I_{\textit{end}}}) - E(L_{i|I_{\textit{begin}}}))},\label{queuelengthintervisit}
\end{equation}
where $L_{i|I_{\textit{begin}}}$ is the number of type $i$ customers at the beginning of an intervisit period $I_i$, and $L_{i|I_{\textit{end}}}$ is the number of type $i$ customers at the end of $I_i$. Since the beginning of an intervisit period coincides with the completion of a visit to $Q_i$, and the end of an intervisit period coincides with the beginning of a visit, we know the GFs for the distributions of these random variables: $\widetilde V_{c_i}(\cdot)$ and $\widetilde V_{b_i}(\cdot)$. This leads to the following expression for the GF of the steady-state queue length distribution of $Q_i$ at an arbitrary epoch, $E[z^{L_i}]$:
\begin{equation}
E[z^{L_i}] = \frac{(1-\rho_i)(1-z)\beta_i(\lambda_i(1-z))}{\beta_i(\lambda_i(1-z))-z} \cdot
\frac{\widetilde{V}_{c_i}(z) - \widetilde{V}_{b_i}(z)}{(1-z)(E(L_{i|I_{\textit{end}}}) - E(L_{i|I_{\textit{begin}}}))}.   \label{queuelengthdecomposition}
\end{equation}
Keilson and Servi \cite{keilsonservi90} show that the distributional form of Little's law can be used to find the LST of the distribution of the marginal waiting time $W_i$: $E(z^{L_i}) = E(\ee^{-\lambda_i(1-z)(W_i + B_i)})$, hence $E(\ee^{-\omega W_i}) = E[(1-\frac{\omega}{\lambda_i})^{L_i}]/\beta_i(\omega)$. This can be substituted into \eqref{queuelengthdecomposition}:
\begin{align}
E[\ee^{-\omega W_i}] &= \frac{(1-\rho_i)\omega}{\omega-\lambda_i(1-\beta_i(\omega))}\cdot
\frac{\widetilde{V}_{c_i}\left(1-\frac{\omega}{\lambda_i}\right) - \widetilde{V}_{b_i}\left(1-\frac{\omega}{\lambda_i}\right)}{(E(L_{i|I_{\textit{end}}}) - E(L_{i|I_{\textit{begin}}}))\omega/\lambda_i}   \nonumber\\
&= E[\ee^{-\omega W_{i|M/G/1}}]E\left[\left(1-\frac\omega{\lambda_i}\right)^{L_{i|I}}\right].\label{waitingtimedecomposition}
\end{align}
The interpretation of this formula is that the waiting time of a type $i$ customer in a polling model is the sum of two independent random variables: the waiting time of a customer in an $M/G/1$ queue with only type $i$ customers, $W_{i|M/G/1}$, and the remaining intervisit time for a customer that arrives at an arbitrary epoch during the intervisit time of $Q_i$.

For \emph{gated} service, the number of type $i$ customers at the beginning of a visit to $Q_i$ is exactly the number of type $i$ customers that arrived during the previous cycle, starting at $Q_i$. In terms of GFs: $\widetilde V_{b_i}(z) = \gamma_i(\lambda_i(1-z))$. The type $i$ customers at the end of a visit to $Q_i$ are exactly those type $i$ customers that arrived during this visit. In terms of GFs: $\widetilde V_{c_i}(z) = \gamma_i(\lambda_i(1-\beta_i(\lambda_i(1-z))))$. We can rewrite $E(L_{i|I_{\textit{end}}}) - E(L_{i|I_{\textit{begin}}})$ as $\lambda_i E(I_i)$, because this is the mean number of type $i$ customers that arrive during an intervisit time. In Section \ref{momentsgeneral} we observe that $\lambda_i E(I_i) = \lambda_i(1-\rho_i)E(C)$. Using these expressions we can rewrite Equation \eqref{waitingtimedecomposition} for gated service to:
\begin{equation}
E[\ee^{-\omega W_i}] = \frac{(1-\rho_i)\omega}{\omega-\lambda_i(1-\beta_i(\omega))}\cdot
\frac{\gamma_i(\lambda_i(1-\beta_i(\omega))) - \gamma_i(\omega)}{(1-\rho_i)\omega E(C)}.\label{lstwgated}
\end{equation}
For \emph{exhaustive} service, $\widetilde V_{c_i}(z) = 1$, because $Q_i$ is empty at the end of a visit to $Q_i$. The number of type $i$ customers at the beginning of a visit to $Q_i$ in an exhaustive polling system is equal to the number of type $i$ customers that arrived during the previous intervisit time of $Q_i$. Hence, $\widetilde{V}_{b_i}(z) = \widetilde{I}_i(\lambda_i(1-z))$, where $\widetilde{I}_i(\cdot)$ is the LST of the intervisit time distribution for $Q_i$. Substitution of $\widetilde{I}_i(\omega) = \widetilde{V}_{b_i}(1-\frac{\omega}{\lambda_i})$ in \eqref{waitingtimedecomposition} leads to the following expression for the LST of the steady-state waiting time distribution of a type $i$ customer in an exhaustive polling system:
\begin{equation}
E[\ee^{-\omega W_i}] = \frac{(1-\rho_i)\omega}{\omega-\lambda_i(1-\beta_i(\omega))}\cdot
\frac{1-\widetilde{I}_i(\omega)}{\omega E(I_i)}.\label{lstwexhaustive}
\end{equation}
In \cite{boonadanboxma2queues2008} another expression is found for $E[\ee^{-\omega W_i}]$. Let the cycle time $C^*_i$ be the time between two successive visit \emph{completions} to $Q_i$. The LST of the cycle time distribution, $\gamma^*_i(\cdot)$, is given by \eqref{cycletimlstVc}. In \cite{boonadanboxma2queues2008} an equation is derived that expresses $\widetilde{I}_i(\cdot)$ in $\gamma^*_i(\cdot)$:
\begin{equation}
\widetilde{I}_i(\omega) = \gamma^*_i(\omega - \lambda_i(1-\beta_i(\omega))).\label{intervisitC}
\end{equation}
Using this expression we can write $E[\ee^{-\omega W_i}]$ in terms of $E[\ee^{-\omega C^*_i}]$:
\begin{align}
E[\ee^{-\omega W_i}] &= \frac{1-\gamma^*_i(\omega - \lambda_i(1-\beta_i(\omega)))}{(\omega-\lambda_i(1-\beta_i(\omega)))E(C)}\nonumber\\
&=  E[\ee^{-(\omega - \lambda_i(1-\beta_i(\omega))) C^*_{i,\textit{res}}}],\label{lstwexhaustiveC}
\end{align}
where $C^*_{i,\textit{res}}$ is the residual length of $C^*_i$.
\begin{remark}
Substitution of  $\omega := s + \lambda_i(1-\pi_i(s))$ in \eqref{lstwexhaustiveC} leads to:
\begin{align*}
E[\ee^{-(s + \lambda_i(1-\pi_i(s))) W_i}] &= E[\ee^{-(s + \lambda_i(1-\pi_i(s)) - \lambda_i(1-\beta_i(s + \lambda_i(1-\pi_i(s))))) C^*_{i,\textit{res}}}]\\
&=E[\ee^{-(s + \lambda_i(1-\pi_i(s)) - \lambda_i(1-\pi_i(s))) C^*_{i,\textit{res}}}]\\
&=E[\ee^{-s C^*_{i,\textit{res}}}].\nonumber
\end{align*}
The interpretation, which can also be found in \cite{boxmaworkloadsandwaitingtimes89}, is that the residual cycle time consists of two components;
firstly the amount of work present at the beginning of the residual cycle, and secondly all work that arrives at $Q_i$ during the residual cycle. By virtue of PASTA, the amount of work at the beginning of the residual cycle is the waiting time of an arbitrary type $i$ customer. The work that arrives during this waiting time, is equal to the length of the busy periods generated by the customers that have arrived during this waiting time.
\end{remark}

\subsection{Moments}\label{momentsgeneral}

The focus of this paper is on LST and GF of distribution functions, not on their moments. Moments can be obtained by differentiation or Taylor series expansion, and are also discussed in \cite{wierman07}. In this subsection we will only mention some results that will be used later.

First we derive the mean cycle time $E(C)$. Unlike higher moments of the cycle time, the mean does not depend on where the cycle starts: $E(C) = \frac{E(S)}{1-\rho}$, where $S = \sum_{i=1}^N S_i$. This can easily be seen, because $1-\rho$ is the fraction of time that the server is not working, but switching. The total mean switch-over time is $\sum_{i=1}^N E(S_i) = E(S)$.

The expected length of a visit to $Q_i$ is $E(V_i) = \rho_i E(C)$, hence the mean length of an intervisit period for $Q_i$ is $E(I_i) = (1-\rho_i)E(C)$. Notice that these expectations do not depend on the service discipline used. The expected number of type $i$ customers at polling moments does depend on the service discipline. For gated service the expected number of type $i$ customers at the beginning of a visit to $Q_i$ is $\lambda_i E(C)$. For exhaustive service this is $\lambda_i E(I_i)$. 

Moments of the waiting time distribution for a type $i$ customer at an arbitrary epoch can be derived from the LSTs given by \eqref{lstwgated}, \eqref{lstwexhaustive} and \eqref{lstwexhaustiveC}. We only present the first moment:
\begin{align}
&Q_i \textrm{ gated: } & E(W_i) &= (1+\rho_i)\frac{E(C_i^2)}{2E(C)},\label{ewgated}\\
&Q_i \textrm{ exhaustive: } & E(W_i) &= \frac{E(I_i^2)}{2E(I_i)}+\frac{\rho_i}{1-\rho_i}\frac{E(B_i^2)}{2E(B_i)}\nonumber\\
& &  &= (1-\rho_i)\frac{E({C^*_i}^2)}{2E(C)}.\label{ewexhaustiveC}
\end{align}
Notice that the start of $C_i$ is the \emph{beginning} of a visit to $Q_i$ for gated service, whereas the start of $C^*_i$ is the \emph{end} of a visit for exhaustive service. Equations \eqref{ewgated} and \eqref{ewexhaustiveC} are in agreement with Equations (4.1) and (4.2) in \cite{boxmaworkloadsandwaitingtimes89}, which also gives interpretations of these equations. Although at first sight these might seem nice, closed formulas, it should be noted that second moments of the cycle time and intervisit time are not easy to determine, requiring the solution of a set of equations. MVA is an efficient technique to compute mean waiting times, the mean residual cycle time, and also the mean residual intervisit time. We refer to \cite{winands06} for an MVA framework for polling models.

\section{The priority polling model with gated and/or exhaustive service}\label{priorities}

In this section we study a polling system with $N$ queues. The service discipline of each queue is either gated, or exhaustive. Each queue contains customers that can be divided into one or more priority classes. Let $K_i$ be the total number of priority levels in $Q_i$. Customers of level 1 receive highest priority, customers of level $K_i$ receive lowest priority. Customers in $Q_i$ of priority level $k\  (k=1,\dots,K_i)$ arrive according to a Poisson process with intensity $\lambda_{ik}$, and have a service requirement $B_{ik}$ with LST $\beta_{ik}(\cdot)$.  If we consider the total number of customers in each queue at polling epochs, which means that we ignore the fact that they will be served according to their priority levels instead of ordinary First-Come-First-Served (FCFS), all the results from Subsections \ref{generaljoint} and \ref{generalcycletime} still hold. In this situation the system should be regarded as a polling system with $N$ queues where customers in $Q_i$ arrive according to a Poisson process with intensity $\lambda_i := \sum_{k=1}^{K_i}\lambda_{ik}$ and have service requirement $B_i$ with LST $\beta_i(\cdot) = \sum_{k=1}^{K_i}\frac{\lambda_{ik}}{\lambda_i} \beta_{ik}(\cdot)$.

We follow the same approach as in Section \ref{general}. First we study the joint queue length distribution of \emph{all} customer types at polling epochs, then the cycle time distribution, followed by the marginal queue length distribution and waiting time distribution, which depend on the service discipline in the selected queue. Subsection \ref{momentspriorities} provides the first moment of these distributions. The last subsection gives a pseudo-conservation law for this system.

\subsection{Joint queue length distribution at polling epochs}

Equation \eqref{p1} gives the GF of the joint queue length distribution at the beginning of a visit to $Q_1$ in a non-priority polling system. Equation \eqref{vbi} can be used to find the GF of the joint queue length distribution at the beginning of a visit to any other queue. The introduction of multiple priority levels to this polling system will only affect the queue length of a queue that is being served, but it will not affect the queue lengths at polling epochs. At the beginning of a visit to $Q_i$, a waiting type $i$ customer has priority level $k$ with probability $\lambda_{ik}/\lambda_{i}$. We can express the GF of the joint queue length distribution in the polling system with priorities, $V^\textit{prio}_{b_i}(z_{11}, \dots, z_{1K_1}, \dots, z_{N1}, \dots, z_{NK_N})$, in terms of the GF of the joint queue length distribution in the polling system without priorities, $V_{b_i}(z_{1},\dots, z_{N})$.

\begin{lemma}
\label{p1_priorities}
\begin{equation}
V^\textit{prio}_{b_i}(z_{11}, \dots, z_{1K_1}, \dots, z_{N1}, \dots, z_{NK_N}) = V_{b_i}\left(\frac{1}{\lambda_1}\sum_{k=1}^{K_1}\lambda_{1k} z_{1k}, \dots, \frac{1}{\lambda_N}\sum_{k=1}^{K_N}\lambda_{Nk} z_{Nk}\right).
\end{equation}
\end{lemma}
\begin{proof}
Let $X_{ijk}$ be the number of customers with priority level $k$ present in $Q_j$ at the beginning of a visit to $Q_i$\ $(i,j=1,\dots,N)$. Let $X_{ij} = \sum_{k=1}^{K_j}X_{ijk}$. Since the type $j$ customers with priority $k$ are exactly those customers that arrived since the previous visit beginning (for gated service; for exhaustive: visit \emph{ending}) at $Q_j$, we know that for $i,j = 1, \dots, N$:
\begin{align*}
&P(X_{ij1}=n_1, X_{ij2}=n_2, \dots, X_{ijK_j} = n_{K_j} | X_{ij} = n) \\
&= \frac{n!}{n_1!\cdots n_{K_j}!}\left(\frac{\lambda_{j1}}{\lambda_{j}}\right)^{n_1}\cdots \left(\frac{\lambda_{jK_j}}{\lambda_{j}}\right)^{n_{K_j}},
\end{align*}
under the condition $\sum_{k=1}^{K_j} n_k = n$.

Hence, for $i,j=1,\dots,N$:
\begin{align*}
&E[z_{j1}^{X_{ij1}} \cdots z_{jK_j}^{X_{ijK_j}} | X_{ij} = n)\\
&= \sum_{n_1=0}^\infty \dots \sum_{n_{K_j}=0}^\infty z_{j1}^{n_{1}} \cdots z_{jK_j}^{n_{K_j}} P(X_{ij1}=n_1, X_{ij2}=n_2, \dots, X_{ijK_j} = n_{K_j} | X_{ij} = n) \\
&= \sum_{\begin{array}{c}n_1,\dots,n_{K_j}\\n_1+\dots+n_{K_j}=n\end{array}}
\frac{n!}{n_1!\cdots n_{K_j}!}\left(\frac{\lambda_{j1}}{\lambda_{j}}z_{j1}\right)^{n_1}\cdots \left(\frac{\lambda_{jK_j}}{\lambda_{j}}z_{jK_j}\right)^{n_{K_j}}\\
&=  
\left(\sum_{k=1}^{K_j} \frac{\lambda_{jk}}{\lambda_{j}}z_{jk}\right)^n.
\end{align*}
Finally,
\begin{align*}
&V^\textit{prio}_{b_i}(z_{11}, \dots, z_{1K_1}, \dots, z_{N1}, \dots, z_{NK_N}) \\
&= \sum_{n_1=0}^\infty \dots \sum_{n_N=0}^\infty \prod_{j=1}^N\left(\sum_{k=1}^{K_j} \frac{\lambda_{jk}}{\lambda_{j}}z_{jk}\right)^{n_j}
P(X_{i1} = n_1, \dots, X_{iN} = n_N) \\
&= V_{b_i}\left(\frac{1}{\lambda_1}\sum_{k=1}^{K_1}\lambda_{1k} z_{1k}, \dots, \frac{1}{\lambda_N}\sum_{k=1}^{K_N}\lambda_{Nk} z_{Nk}\right),
\qquad\qquad i=1,\dots,N.
\end{align*}
\end{proof}

\subsection{Cycle time}

The LST of the cycle time distribution is still given by \eqref{cycletimelst} (cycle starting at a visit beginning) or \eqref{cycletimlstVc} (cycle starting at a visit completion) if we define $\lambda_i := \sum_{k=1}^{K_i}\lambda_{ik}$ and $\beta_i(\cdot) := \sum_{k=1}^{K_i}\frac{\lambda_{ik}}{\lambda_i} \beta_{ik}(\cdot)$, for $i=1,\dots,N$, because the cycle time does not depend on the order of service within each queue.

\subsection{Marginal queue lengths and waiting times}\label{gatedmarginalw}

The LST of the marginal queue length distribution of $Q_i$ at an arbitrary epoch depends on the service discipline that is used to serve $Q_i$, and so does the LST of the waiting time distribution. First we study a queue with gated service, afterwards a queue with exhaustive service.

\subsubsection*{Gated service}

 We focus on a particular queue, say $Q_i$, assuming it has gated service. We determine the LST of the waiting time distribution for a type $i$ customer with priority level $k$, from now on called a ``type $ik$ customer'', using the fact that this customer will not be served until the next cycle (starting at a visit \emph{beginning} to $Q_i$). The time from the start of the cycle until the arrival of this customer will be called ``past cycle time'', denoted by $C_{iP}$. The residual cycle time will be denoted by $C_{iR}$. The waiting time of a type $ik$ customer is composed of $C_{iR}$, the service times of all higher priority customers that arrived during $C_{iP}+C_{iR}$, and the service times of all type $ik$ customers that have arrived during $C_{iP}$. Let $L_H(T)$ be the number of higher priority customers that have arrived during time interval $T$, and equivalently define $L_k(T)$ for the type $ik$ customers. The arrival process of higher priority customers is a Poisson process with intensity $\lambda_{H} := \sum_{j=1}^{k-1} \lambda_{ij}$. This intensity equals 0 if $k=1$, which corresponds to the highest priority customers in $Q_i$. The service time distribution of these higher priority customers is denoted by $B_{H}$ (suppressing $i$ and $k$ for notational reasons), with LST $\beta_{H}(\cdot) := \sum_{j=1}^{k-1}\frac{\lambda_{ij}}{\lambda_{H}} \beta_{ij}(\cdot)$.

\begin{theorem}
The LST of the waiting time distribution of a type $ik$ customer in a polling system with $Q_i$ being served according to the gated discipline, is:
\[
E\left[\ee^{-\omega W_{ik}}\right] =\frac{\gamma_i\big(\lambda_H(1-\beta_H(\omega))+\lambda_{ik}(1-\beta_{ik}(\omega))\big) - \gamma_i\big(\omega+\lambda_H(1-\beta_H(\omega))\big)}{[\omega-\lambda_{ik}(1-\beta_{ik}(\omega))]E(C)}.
\]
\end{theorem}

\begin{proof}
\begin{align}
E\left[\ee^{-\omega W_{ik}}\right] =& E\left[\ee^{-\omega (C_{iR}+\sum_{j=1}^{L_H(C_{iP}+C_{iR})}B_{H,j}+\sum_{j=1}^{L_k(C_{iP})}B_{ik,j})}\right] \nonumber\\
=& \int_{t=0}^\infty \int_{u=0}^\infty \sum_{m=0}^\infty\sum_{n=0}^\infty E\left[\ee^{-\omega (u+\sum_{j=1}^{m}B_{H,j}+\sum_{j=1}^{n}B_{ik,j})}\right] \nonumber\\
&\cdot P(L_H(C_{iP}+C_{iR})=m, L_k(C_{iP})=n) \dd P(C_{iP}<t, C_{iR}<u)\nonumber\\
=& \int_{t=0}^\infty \int_{u=0}^\infty \ee^{-\omega u} \sum_{m=0}^\infty\sum_{n=0}^\infty E\left[\ee^{-\omega\sum_{j=1}^{m}B_{H,j}}\right]E\left[\ee^{-\omega\sum_{j=1}^{n}B_{ik,j}}\right]
\nonumber\\
&\cdot\frac{(\lambda_H(t+u))^m}{m!}\ee^{-\lambda_H(t+u)}\frac{(\lambda_{ik} t)^n}{n!}\ee^{-\lambda_{ik}t}
\dd P(C_{iP}<t, C_{iR}<u)\nonumber\\
\nonumber\\
=& \int_{t=0}^\infty \int_{u=0}^\infty \ee^{-t\big(\lambda_H(1-\beta_H(\omega))+\lambda_{ik}(1-\beta_{ik}(\omega))\big)}
\ee^{-u\big(\omega + \lambda_H(1-\beta_H(\omega))\big)} \dd P(C_{iP}<t, C_{iR}<u)\nonumber\\
=&\frac{\gamma_i\big(\lambda_H(1-\beta_H(\omega))+\lambda_{ik}(1-\beta_{ik}(\omega))\big) - \gamma_i\big(\omega+\lambda_H(1-\beta_H(\omega))\big)}{[\omega-\lambda_{ik}(1-\beta_{ik}(\omega))]E(C)}.
\label{lstwikgated}
\end{align}
For the last step in the derivation of \eqref{lstwikgated} we used
\[E[\ee^{-\omega_P C_{iP}-\omega_R C_{iR}}] = \frac{E[\ee^{-\omega_P C_i}]-E[\ee^{-\omega_R C_i}]}{(\omega_R-\omega_P)E(C)},\]
see, e.g., \cite{boxmalevyyechiali92}. The cycle time $C_i$ starts with a visit \emph{beginning} to $Q_i$, and its LST $\gamma_i(\cdot)$ is given by \eqref{cycletimelst}.
\end{proof}

We now present an alternative derivation of \eqref{lstwikgated} by modelling the system in such a way that the Fuhrmann-Cooper decomposition \cite{fuhrmanncooper85} can be used. Fuhrmann and Cooper \cite{fuhrmanncooper85} showed that the waiting time of a customer in an $M/G/1$ queue with server vacations is the sum of two independent quantities: the waiting time of a customer in a corresponding $M/G/1$ queue without vacations, and the residual vacation time. Conditions that have to be satisfied in order to use this composition, are:
\begin{enumerate}
\item Poisson arrivals, service times are independent of each other, of the arrival process and of the sequence of vacation periods that precede a certain service time;
\item all customers are eventually served;
\item customers are served in an order that is independent of their service times;
\item service is nonpreemptive;
\item the rules that govern when the server begins and ends vacations do not anticipate future jumps of the Poisson arrival process.
\end{enumerate}
In order to use the Fuhrmann-Cooper decomposition to determine $E\left[\ee^{-\omega W_{ik}}\right]$, we first model the system as a polling system with $N+2$ queues $(Q_1, \dots, Q_{i-1}, Q_H, Q_{ik}, Q_L, Q_{i+1}, \dots, Q_N)$ and no switch-over times between $Q_H, Q_{ik}$, and $Q_L$. $Q_H$ contains all type $i$ customers with priority level higher than $k$ (i.e. priority \emph{index} less than $k$), $Q_{ik}$ contains all type $i$ customers of priority level $k$, and $Q_L$ contains all type $i$ customers of priority level lower than $k$. Lower priority customers arrive according to a Poisson process with intensity $\lambda_{L} := \sum_{j=k+1}^{K_i} \lambda_{ij}$ (again suppressing index $i$ for notational reasons). The service time distribution of these customers is denoted by $B_{L}$, with LST $\beta_{L}(\cdot) := \sum_{j=k+1}^{K_i}\frac{\lambda_{ij}}{\lambda_{L}} \beta_{ij}(\cdot)$.
The service discipline of this equivalent system is synchronised gated, which is a more general version of gated. The gates for queues $Q_H, Q_{ik}$, and $Q_L$ are set simultaneously when the server arrives at $Q_H$, but the gates for the other queues are still set at the server's arrival at these queues. From the viewpoint of a type $ik$ customer, and as far as waiting times are concerned, this system is an $M/G/1$ queue with generalised vacations (because service is not exhaustive and the length of a vacation is positively correlated with the length of a visit period). The vacation is the intervisit period of $Q_{ik}$, which starts with a service to $Q_L$ and ends when service of $Q_H$ is completed. Fuhrmann and Cooper \cite{fuhrmanncooper85} state that the number of type $ik$ customers in the system at an arbitrary epoch ($L_{ik}$) is the sum of the queue length in an isolated $M/G/1$ queue with only type $ik$ customers ($L_{ik|M/G/1}$), and the number of type $ik$ customers at a random epoch during a vacation ($L_{ik|I}$):
\begin{equation}
E\left[z^{L_{ik}}\right] = \frac{(1-\rho_{ik})(1-z)\beta_{ik}(\lambda_{ik}(1-z))}{\beta_{ik}(\lambda_{ik}(1-z))-z} \cdot \frac{\widetilde{V}_{c_{ik}}(z) - \widetilde{V}_{b_{ik}}(z)}{(1-z)(E(L_{{ik}|I_{\textit{end}}}) - E(L_{ik|I_{\textit{begin}}}))}.\label{gfzikgateddecomposition}
\end{equation}
Equation \eqref{queuelengthintervisit} has been used to derive the second term at the right-hand side of \eqref{gfzikgateddecomposition}, where $\widetilde{V}_{b_{ik}}(z)$ denotes the GF of the distribution of the number of type $ik$ customers at the beginning of a visit to $Q_{ik}$, and $\widetilde{V}_{c_{ik}}(z)$ denotes the GF at the completion of a visit to $Q_{ik}$.
The number of type $ik$ customers at the beginning of a visit to $Q_{ik}$ is the number of type $ik$ customers that arrived during the previous cycle (starting with a visit to $Q_i$, which corresponds to a visit beginning to $Q_H$ in this model), plus the number of type $ik$ customers that arrived during the service of the type $H$ customers that arrived during the previous cycle. However, because of the synchronised gated service discipline only those type $ik$ customers present at the visit beginning of $Q_{H}$ will be served:
\begin{align*}
\widetilde{V}_{b_{ik}}(z) &= \gamma_i\big(\lambda_H(1-\beta_H(\lambda_{ik}(1-z))) + \lambda_{ik}(1-z)\big), \\
\widetilde{V}_{c_{ik}}(z) &= \gamma_i\big(\lambda_H(1-\beta_H(\lambda_{ik}(1-z))) + \lambda_{ik}(1-\beta_{ik}(\lambda_{ik}(1-z)))\big).
\end{align*}
Furthermore, we have $E(L_{ik|I_{\textit{end}}}) - E(L_{ik|I_{\textit{begin}}}) = \lambda_{ik} E(I_{ik}) = \lambda_{ik}(1-\rho_{ik}) E(C)$. Now, to derive the LST of the waiting time $W_{ik}$, we can apply the distributional form of Little's law \cite{keilsonservi90} to \eqref{gfzikgateddecomposition}, because the required conditions are fulfilled for each customer class (including customer types $H$, $ik$, and $L$): the customers enter the system in a Poisson stream, every customer enters the system and leaves the system one at a time in order of arrival, and for any time $t$ the entry process into the system of customers after time $t$ and the time spent in the system by any customer arriving before time $t$ are independent. This leads to the following expression:
\begin{align}
E\left[\ee^{-\omega W_{ik}}\right] =\,& \frac{(1-\rho_{ik})\omega}{\omega-\lambda_{ik}(1-\beta_{ik}(\omega))} \nonumber\\
&\cdot
\frac{\gamma_i\big(\lambda_H(1-\beta_H(\omega))+\lambda_{ik}(1-\beta_{ik}(\omega))\big) - \gamma_i\big(\omega+\lambda_H(1-\beta_H(\omega))\big)}{(1-\rho_{ik})\omega E(C)}.\label{lstwikgateddecomposition}
\end{align}
We recognise the first term on the right-hand side of \eqref{lstwikgateddecomposition} as the LST of the waiting time distribution of an $M/G/1$ queue with only type $ik$ customers. Equation \eqref{lstwikgateddecomposition} can be rewritten to \eqref{lstwikgated}.

\subsubsection*{Exhaustive service}

Analysis of the model with exhaustive service requires a different approach. The key observation, made by Fuhrmann and Cooper \cite{fuhrmanncooper85}, is that the polling system from the viewpoint of a type $i$ customer is an $M/G/1$ queue with multiple server vacations. The $M/G/1$ queue with priorities and vacations has been extensively analysed by Kella and Yechiali \cite{kellayechiali88}. We use their approach to find the waiting time LST for type $ik$ customers. Kella and Yechiali \cite{kellayechiali88} distinguish between systems with single and multiple vacations, and preemptive resume and nonpreemptive service. The polling system that we consider corresponds to the system with multiple vacations, because the server is never idling, but keeps on switching if the system is empty. In the present paper we do not focus on preemptive resume, but we will give some results on it later in this subsection. First we focus on the case labelled as NPMV (nonpreemptive, multiple vacations) in \cite{kellayechiali88}. We consider the system from the viewpoint of a type $ik$ customer to derive $E[\ee^{-\omega W_{ik}}]$. For the waiting time of a type $ik$ customer, the order in which the higher priority customers are served is irrelevant. We will refer to all type $i$ customers with higher priority than $k$ as ``type $H$ customers'' arriving according to one Poisson stream with intensity $\lambda_{H} := \sum_{j=1}^{k-1} \lambda_{ij}$. The service time distribution of these customers is denoted by $B_{H}$, with LST $\beta_{H}(\cdot) := \sum_{j=1}^{k-1}\frac{\lambda_{ij}}{\lambda_{H}} \beta_{ij}(\cdot)$. A busy period of type $H$ customers is denoted by $\textit{BP}_H$ with LST $\pi_H(\cdot)$ which is the root of the equation $\pi_H(\omega) = \beta_H(\omega+\lambda_H(1-\pi_H(\omega)))$. In a similar way we define ``type $L$'' customers which arrive according to a Poisson stream with intensity $\lambda_{L} := \sum_{j=k+1}^{K_i} \lambda_{ij}$. The service time of these customers is denoted by $B_{L}$, with LST $\beta_{L}(\cdot) := \sum_{j=k+1}^{K_i}\frac{\lambda_{ij}}{\lambda_{L}} \beta_{ij}(\cdot)$.

From the viewpoint of a type $ik$ customer and as far as waiting times are considered, a polling system is a \emph{nonpriority} single server system with multiple vacations. The vacation can either be the intervisit period $I_i$, or a busy period of type $H$ customers, or the service of a type $L$ customer. The LSTs of these three types of vacations are:
\begin{align}
E[\ee^{-\omega I_i}] &= \widetilde{V}_{b_i}(1-\omega/\lambda_i), \nonumber\\
E[\ee^{-\omega \textit{BP}_H}] &= \pi_H(\omega),\label{intervisitexhaustive}\\
E[\ee^{-\omega B_L}] &= \beta_L(\omega). \nonumber
\end{align}
The first equality in Equation \eqref{intervisitexhaustive} follows immediately from the fact that the number of type $i$ customers at the beginning of a visit to $Q_i$ is the number of type $i$ customers that have arrived during the previous intervisit period: $\widetilde{V}_{b_i}(z) = E[\ee^{-(\lambda_i(1-z)) I_i}]$.

The key observation is that an arrival of a type $ik$ customer will always take place within either an $I_{H,ik}$ cycle, or an $L_{H,ik}$ cycle. An $I_{H,ik}$ cycle is a cycle that starts with an intervisit period for $Q_i$, followed by the service of all customers with priority level $H$ or $k$, and ends at the moment that no type $H$ or $ik$ customers are left in the system. Notice that at the start of the intervisit period, no type $H$ or $ik$ customers were present in the system either. An $L_{H,ik}$ cycle is a similar cycle, but starts with the service of a type $L$ customer. This cycle also ends at the moment that no type $H$ or $ik$ customers are left in the system.

The fraction of time that the system is in an $L_{H,ik}$ cycle is $\frac{\rho_L}{1-\rho_H-\rho_{ik}}$, because type $L$ customers arrive with intensity $\lambda_L$. Each of these customers will start an $L_{H,ik}$ cycle and the length of an $L_{H,ik}$ cycle equals $\frac{E(B_L)}{1-\rho_H-\rho_{ik}}$:
\begin{align*}
E(L_{H,ik}\textrm{ cycle}) &= E(B_L) + \lambda_{H,ik} E(B_L) E(\textit{BP}_{H,ik}) \\
&= E(B_L) + \lambda_{H,ik} E(B_L) \frac{E(B_{H,ik})}{1-\rho_{H,ik}} \\
&= \left(1+\frac{\rho_H+\rho_{ik}}{1-\rho_H-\rho_{ik}}\right)E(B_L) = \frac{E(B_L)}{1-\rho_H-\rho_{ik}},
\end{align*}
where $\lambda_{H,ik}, B_{H,ik}, \textit{BP}_{H,ik}$, and $\rho_{H,ik}$ denote the intensity, service time, busy period and occupation rate corresponding to type $i$ customers with priority level $1,\dots, k$.

The fraction of time that the system is in an $I_{H,ik}$ cycle, is $1-\frac{\rho_L}{1-\rho_H-\rho_{ik}} = \frac{1-\rho_i}{1-\rho_H-\rho_{ik}}$. This result can also be obtained by using the argument that the fraction of time that the system is in an intervisit period is the fraction of time that the server is not serving $Q_i$, which is equal to $1-\rho_i$. A cycle which starts with such an intervisit period and stops after the service of all customers with priority levels $1,\dots,k$ that arrived during the intervisit period and their descendants of priority levels $1,\dots,k$, has mean length $E(I_i) + \lambda_{H,ik} E(I_i) E(\textit{BP}_{H,ik}) = \frac{E(I_i)}{1-\rho_H-\rho_{ik}}$. This also leads to the conclusion that  $\frac{1-\rho_i}{1-\rho_H-\rho_{ik}}$ is the fraction of time that the system is in an $I_{H,ik}$ cycle.

A type $ik$ customer arriving during an $I_{H,ik}$ cycle views the system as a nonpriority $M/G/1$ queue with multiple server vacations. The service requirement in this system is \emph{not} the service requirement of a type $ik$ customer, $B_{ik}$, but the time that is required to serve a type $ik$ customer and all higher priority customers that arrive during this service until no higher priority customers are present in the system. We will elaborate on this so-called \emph{completion time} later. The vacation is the intervisit time $I_i$, plus the service times of all type $H$ customers that have arrived during that intervisit time and their type $H$ descendants. We denote this extended intervisit time by $I_i^*$ with LST
\[\widetilde{I}_i^*(\omega) = \widetilde{I}_i(\omega+\lambda_H(1-\pi_H(\omega))).\]
The mean of $I_i^*$ equals $E(I_i^*) = \frac{E(I_i)}{1-\rho_H}$.

A type $ik$ customer arriving during an $L_{H,ik}$ cycle views the system as a nonpriority $M/G/1$ queue with multiple server vacations $B_L$ plus the busy periods generated by all type $H$ customers that have arrived during the service time of the type $L$ customer. We will denote this extended service time by $B_L^*$ with mean $E(B_L^*)= \frac{E(B_L)}{1-\rho_H}$, and LST
\[\beta_L^*(\omega) = \beta_L(\omega+\lambda_H(1-\pi_H(\omega))).\]
We also have to take into account that a busy period of type $ik$ customers might be interrupted by the arrival of type $H$ customers. Therefore, in the alternative system that we are considering, the service time of a type $ik$ customer is equal to $B_{ik}$ plus the service times of all type $H$ customers that arrive during this service time, and all of their type $H$ descendants. The LST of the distribution of this extended service time $B_{ik}^*$ is
\[\beta_{ik}^*(\omega) = \beta_{ik}(\omega + \lambda_H(1-\pi_H(\omega))).\]
This extended service time is often called completion time in the literature, cf. \cite{takagi90}. In this alternative system, the mean service time of these customers equals $E(B_{ik}^*) = \frac{E(B_{ik})}{1-\rho_H}$. The fraction of time that the system is serving these customers is $\rho_{ik}^* = \frac{\rho_{ik}}{1-\rho_H} = 1 - \frac{1-\rho_H-\rho_{ik}}{1-\rho_H}$.

\begin{theorem}
The LST of the waiting time distribution of a type $ik$ customer in polling system with $Q_i$ being served exhaustively, is:
\begin{align}
E[\ee^{-\omega W_{ik}}] &= \frac{(1-\rho_{ik}^*)\omega}{\omega-\lambda_{ik}(1-\beta_{ik}^*(\omega))} \cdot
\left[\frac{1-\rho_i}{1-\rho_H-\rho_{ik}} \cdot \frac{1-\widetilde{I}_i^*(\omega)}{\omega E(I_i^*)} + \frac{\rho_L}{1-\rho_H-\rho_{ik}}\cdot\frac{1-\beta_L^*(\omega)}{\omega E(B_L^*)}\right].\label{lstwikexhaustive}
\end{align}
\end{theorem}

\begin{proof}
The system has been modelled as an $M/G/1$ queue with customers that have service requirement $B_{ik}^*$, and multiple server vacations that have length $I^*_i$ with probability $\frac{1-\rho_i}{1-\rho_H-\rho_{ik}}$, and length $B_L^*$ with probability $\frac{\rho_L}{1-\rho_H-\rho_{ik}}$. Equation \eqref{lstwikexhaustive} follows immediately from application of the Fuhrmann-Cooper decomposition to this system.
\end{proof}
\begin{remark}
Equation \eqref{lstwikexhaustive} can be rewritten in several forms, which lead to different interpretations. A compact form is
\begin{equation}
E[\ee^{-\omega W_{ik}}]= \frac{\frac{1-\rho_i}{E(I_i)}[1-\widetilde{I}_i(\omega+\lambda_H-\lambda_H\pi_H(\omega))]+\lambda_L[1-\beta_L(\omega+\lambda_H-\lambda_H\pi_H(\omega))]}
{\lambda_{ik}\beta_{ik}(\omega+\lambda_H-\lambda_H\pi_H(\omega))-\lambda_{ik}+\omega},\label{lstwikexhaustivecompact}
\end{equation}
which is the form used by Kella and Yechiali (see \cite{kellayechiali88}, Section 4.1).

Another form, which is more suitable for interpretation, is
\begin{equation}
\begin{aligned}
E[\ee^{-\omega W_{ik}}] &=
\frac{(1-\rho_{ik}^*)\omega}{\omega-\lambda_{ik}(1-\beta_{ik}^*(\omega))} \\
&\cdot \left[\frac{1-\rho_i}{1-\rho_H-\rho_{ik}} \cdot\frac{1-\widetilde{I}_i(\omega+\lambda_H(1-\pi_H(\omega)))}{(\omega+\lambda_H(1-\pi_H(\omega))) E(I_i)} \right. \\
&\qquad+\left.\frac{\rho_L}{1-\rho_H-\rho_{ik}}\cdot\frac{1-\beta_L(\omega+\lambda_H(1-\pi_H(\omega)))}{(\omega+\lambda_H(1-\pi_H(\omega)))E(B_L)}\right]\\
&\cdot \left[(1-\rho_H) + \rho_H \cdot \frac{1-\pi_H(\omega)}{\omega E(\textit{BP}_H)}\right].
\end{aligned}
\label{lstwikexhaustivedecomp}
\end{equation}

The first term on the right hand side of \eqref{lstwikexhaustivedecomp} is the LST of the waiting time distribution of a customer in a nonpriority $M/G/1$ queue with customers arriving with intensity $\lambda_{ik}$ and having service requirement LST $\beta_{ik}^*(\cdot)$, i.e. the completion time of a type $ik$ customer. The second term indicates that with probability $\frac{1-\rho_i}{1-\rho_H-\rho_{ik}}$ a type $ik$ customer has to wait a residual intervisit time, plus the busy periods of all higher priority customers that arrive during this residual intervisit time. With probability $\frac{\rho_L}{1-\rho_H-\rho_{ik}}$ the customer has to wait a residual service time of a type $L$ customer, plus the busy periods of all higher priority customers that arrive during this residual intervisit time. The last term in \eqref{lstwikexhaustivedecomp} indicates that with probability $\rho_H$ a customer also has to wait for the residual length of a busy period of higher priority customers.
\end{remark}

\begin{remark}
Substitution of \eqref{intervisitC} in \eqref{lstwikexhaustivecompact} leads to a different expression for $E[\ee^{-\omega W_{ik}}]$:
\begin{align}
&E[\ee^{-\omega W_{ik}}]\nonumber\\
&= \frac{1-\gamma_i^*\big(\omega
-\lambda_{ik}(1-\beta_{ik}(\omega+\lambda_H(1-\pi_H(\omega))))
-\lambda_{L}(1-\beta_{L}(\omega+\lambda_H(1-\pi_H(\omega))))
\big)}
{[\omega-\lambda_{ik}(1-\beta_{ik}(\omega+\lambda_H(1-\pi_H(\omega))))]E(C)}\nonumber
\\
&\quad+\frac{\lambda_L\big[1-\beta_L(\omega+\lambda_H(1-\pi_H(\omega)))\big]}
{\omega-\lambda_{ik}(1-\beta_{ik}(\omega+\lambda_H(1-\pi_H(\omega))))}.\label{lstwikexhaustiveC}
\end{align}
For the lowest priority customers this expression simplifies to
\begin{align*}
E[\ee^{-\omega W_{iK_i}}] &=\frac{1-\gamma^*_i\big(\omega - \lambda_{iK_i}(1-\beta_{iK_i}(\omega+\lambda_H(1-\pi_H(\omega))))\big)}{\big(\omega-\lambda_{iK_i}(1-\beta_{iK_i}(\omega+\lambda_H(1-\pi_H(\omega))))\big)E(C)}\\
&=E[\ee^{-(\omega - \lambda_{iK_i}(1-\beta_{iK_i}(\omega+\lambda_H(1-\pi_H(\omega)))))C^*_{i,\textit{res}}}],
\end{align*}
where type $H$ customers in this case are all type $i$ customers except for those with the lowest priority.
\end{remark}

\begin{remark}
In the derivation of Equation \eqref{lstwikexhaustive} we followed the approach of Kella and Yechiali \cite{kellayechiali88}, but restricted ourselves to the nonpreemptive situation only. 
Using the decomposition \eqref{lstwikexhaustivedecomp}, it only requires a small extra step to obtain the waiting time LST in a system with \emph{preemptive resume}. The term $\frac{1-\beta_L(\omega+\lambda_H(1-\pi_H(\omega)))}{(\omega+\lambda_H(1-\pi_H(\omega)))E(B_L)}$ on the third line is the LST of a residual service time of a type $L$ customer, plus the busy periods of all type $H$ customers that arrive during this residual type $L$ service time. If type $ik$ customers are allowed to preempt the service of a type $L$ customer, this term simply vanishes from the waiting time LST. This leads to the following expression for the waiting time LST of a type $ik$ customer if the service policy is \emph{preemptive resume}:
\begin{equation}
\begin{aligned}
E[\ee^{-\omega W_{ik}}] &=
\frac{(1-\rho_{ik}^*)\omega}{\omega-\lambda_{ik}(1-\beta_{ik}^*(\omega))} \cdot \left[(1-\rho_H) + \rho_H \cdot \frac{1-\pi_H(\omega)}{\omega E(\textit{BP}_H)}\right]\\
&\cdot \left[\frac{1-\rho_i}{1-\rho_H-\rho_{ik}} \cdot\frac{1-\widetilde{I}_i(\omega+\lambda_H(1-\pi_H(\omega)))}{(\omega+\lambda_H(1-\pi_H(\omega))) E(I_i)} +\frac{\rho_L}{1-\rho_H-\rho_{ik}}\right]. \end{aligned}
\label{lstwikexhaustivepreemptiveresume}
\end{equation}
Note that the \emph{sojourn time} of a type $ik$ customer equals the waiting time plus a \emph{completion time} in this situation. Equation \eqref{lstwikexhaustivepreemptiveresume} is in agreement with the result in \cite{kellayechiali88}.

\end{remark}

We will refrain from mentioning the GFs of the marginal queue length distributions here, because they can be obtained by applying the distributional form of Little's law as we have done before.

\subsection{Moments}\label{momentspriorities}

As mentioned in Section \ref{momentsgeneral}, we do not focus on moments in this paper, and we only mention the mean waiting times of type $ik$ customers. The formulas in this subsection can also be obtained using MVA, as shown in \cite{wierman07}, but we have obtained them by differentiating \eqref{lstwikgated}, \eqref{lstwikexhaustive}, and \eqref{lstwikexhaustivepreemptiveresume} respectively:
\begin{align*}
&Q_i \textrm{ gated: } & E(W_{ik}) &= (1+2\rho_H+\rho_{ik})E(C_{i,\textit{res}})\\
&&&= (1+2\sum_{j=1}^{k-1}\rho_{ij}+\rho_{ik})E(C_{i,\textit{res}}),\label{ewikgated}\\
&\parbox{4cm}{$Q_i$ exhaustive \\(nonpreemptive): } & E(W_{ik}) &=
\frac{\sum_{j=1}^{K_i} \rho_{ij} E(B_{ij,\textit{res}})+(1-\rho_i)E(I_{i,\textit{res}})}{(1-\sum_{j=1}^{k-1}\rho_{ij})(1-\sum_{j=1}^{k}\rho_{ij})},
\\
&\parbox{4cm}{$Q_i$ exhaustive \\(preemptive resume): } & E(W_{ik}) &= \frac{\sum_{j=1}^{k} \rho_{ij} E(B_{ij,\textit{res}})+(1-\rho_i)E(I_{i,\textit{res}})}{(1-\sum_{j=1}^{k-1}\rho_{ij})(1-\sum_{j=1}^{k}\rho_{ij})}.
\end{align*}
Differentiation of \eqref{lstwikexhaustiveC} leads an to alternative expression for the mean waiting time of a type $ik$ customer if $Q_i$ is served exhaustively:
\begin{align*}
E(W_{ik}) &= \frac{(1-\rho_i)^2}{(1-\rho_H)(1-\rho_H-\rho_{ik})}\frac{E({C^*_i}^2)}{2E(C)}\\
&=\frac{(1-\rho_i)^2}{(1-\sum_{j=1}^{k-1}\rho_{ij})(1-\sum_{j=1}^{k}\rho_{ij})}\frac{E({C^*_i}^2)}{2E(C)}.
\end{align*}
Notice that the cycle time in this exhaustive case starts at a visit \emph{completion} to $Q_i$.

\subsection{A pseudo-conservation law}\label{pseudoconservationlawsection}

Let $V$ be the amount of work in the polling system at an arbitrary epoch. Boxma and Groenendijk \cite{boxmagroenendijk87} show that $V$ can be written as the sum of two independent random variables:
\begin{equation}
V =^{\hspace{-0.5em}\mbox{}^d} V_{M/G/1}+V_I,\label{pseudoconservationlawnonpriority}
\end{equation}
with $V_{M/G/1}$ the amount of work at an arbitrary epoch in the corresponding $M/G/1$ queue, and $V_I$ the amount of work in the system at an arbitrary epoch in an intervisit period when the server is idling. It is shown in \cite{boxmagroenendijk87} that relation \eqref{pseudoconservationlawnonpriority} leads to a \emph{pseudo-conservation law} for nonpriority polling systems which is not specified in more detail here. This law has been generalised for polling systems with multiple priority levels in each queue in \cite{shimogawatakahashi88} and \cite{fournierrosberg91}:
\begin{equation}
\begin{aligned}
\sum_{i=1}^N\sum_{k=1}^{K_i} \rho_{ik} E(W_{ik}) &=
\frac{\rho}{1-\rho} \sum_{i=1}^{N}\sum_{k=1}^{K_i} \rho_{ik}\frac{E(B_{ik}^2)}{2E(B_{ik})}\\
& + \rho \frac{E(S^2)}{2E(S)} +\left[\rho^2-\sum_{i=1}^N \rho_i^2\right]\frac{E(S)}{2(1-\rho)} + \sum_{i=1}^N E(Z_{ii}),
\end{aligned}
\label{pseudoconservationlawpriorities}
\end{equation}
where $S = \sum_{i=1}^N S_i$ and $Z_{ii}$ is the amount of work left behind by the server at $Q_i$ at the completion of a visit. Equation \eqref{pseudoconservationlawpriorities} is called a pseudo-conservation law because it indicates that the weighted sum of the mean waiting times $\sum_{i=1}^N\sum_{k=1}^{K_i} \rho_{ik} E(W_{ik})$ is constant, except for the term $Z_{ii}$ which depends on the service discipline. For gated service, $E(Z_{ii}) = \rho_i^2 E(C)$. For exhaustive service, $E(Z_{ii}) = 0$ because $Q_i$ is empty at the completion of a visit to this queue.
For globally gated service, which is discussed in the next section, $E(Z_{ii}) = \rho_i\left(E(C)\sum_{j=1}^i \rho_j +\sum_{j=1}^{i-1} E(S_j)\right)$.

\section{Globally gated service}\label{globallygated}

In this section we discuss a polling model with $N$ queues $(Q_1, \dots, Q_N)$ and $K_i$ priority classes in $Q_i$ with globally gated service. For this service discipline, only customers that were present when the server started its visit to $Q_1$ are served. This feature makes the model exactly the same as a nonpriority polling model with $\sum_{i=1}^N K_i$ queues $(Q_{11}, \dots, Q_{1K_1}, \dots, Q_{N1}, \dots, Q_{NK_N})$.

It was observed in Section \ref{general} that, although this system does not satisfy Property \ref{resingproperty}, it does satisfy Property \ref{borstproperty} which implies that we can still follow the same approach as in the previous sections.

\subsection{Joint queue length distribution at polling epochs}

We define the beginning of a visit to $Q_1$ as the start of a cycle, since this is the moment that determines which customers will be served during the next visits to the queues. Arriving customers will always be served in the next cycle, so the offspring GFs are:
\begin{align*}
&f^{(ik)}(z_{11}, \dots, z_{1K_1}, \dots, z_{N1}, \dots, z_{NK_N}) \\
&\,= h_{ik}(z_{11}, \dots, z_{1K_1}, \dots, z_{N1}, \dots, z_{NK_N}) \\
&\,= \beta_{ik}(\sum_{j=1}^N\sum_{l=1}^{K_j}\lambda_{jl}(1-z_{jl})),\qquad i=1,\dots,N; k=1,\dots,K_i.
\end{align*}
The $N$ immigration functions are:
\[
g^{(i)}(z_{11}, \dots, z_{1K_1}, \dots, z_{N1}, \dots, z_{NK_N}) = \sigma_i(\sum_{j=1}^N\sum_{l=1}^{K_j}\lambda_{jl}(1-z_{jl})),\qquad\qquad i=1,\dots,N.
\]
Using these definitions, the formula for the GF of the joint queue length distribution at the beginning of a cycle is similar to the one found in Section \ref{general}:
\begin{equation}
P_1(z_{11}, \dots, z_{1K_1}, \dots, z_{N1}, \dots, z_{NK_N}) = \prod_{n=0}^\infty g(f_n(z_{11}, \dots, z_{1K_1}, \dots, z_{N1}, \dots, z_{NK_N})).
\end{equation}

Notice that in a system with globally gated service it is possible to express the joint queue length distribution at the beginning of a cycle in terms of the cycle time LST, since all customers that are present at the beginning of a cycle are exactly all of the customers that have arrived during the previous cycle:

\begin{equation}
P_1(z_{11}, \dots, z_{1K_1}, \dots, z_{N1}, \dots, z_{NK_N}) = \gamma_1(\sum_{i=1}^N\sum_{k=1}^{K_i}\lambda_{ik}(1-z_{ik})).\label{p1globallygated}
\end{equation}

\subsection{Cycle time}

Since only those customers that are present at the start of a cycle, starting at $Q_1$, will be served during this cycle, the LST of the cycle time distribution is
\begin{equation}
\gamma_1(\omega) = \prod_{i=1}^N\sigma_i(\omega) P_1(\beta_{11}(\omega), \dots, \beta_{1K_1}(\omega), \dots, \beta_{N1}(\omega), \dots, \beta_{NK_N}(\omega)).\label{lstcgloballygated}
\end{equation}
Substitution of \eqref{p1globallygated} into this expression gives us the following relation:
\[
\gamma_1(\omega) = \prod_{i=1}^N\sigma_i(\omega)\,\gamma_1(\sum_{i=1}^N\sum_{k=1}^{K_i}\lambda_{ik}(1-\beta_{ik}(\omega))).
\]
Boxma, Levy and Yechiali \cite{boxmalevyyechiali92} show (for a model without priorities) that this relation leads to the following expression for the cycle time LST:
\[
\gamma_1(\omega) = \prod_{i=0}^\infty \sigma(\delta^{(i)}(\omega)),
\]
where $\sigma(\cdot) = \prod_{i=1}^N\sigma_i(\cdot)$, and $\delta^{(i)}(\omega)$ is recursively defined as follows:
\begin{align*}
\delta^{(0)}(\omega) &= \omega,\\
\delta^{(i)}(\omega) &= \delta(\delta^{(i-1)}(\omega)), \qquad\qquad i=1,2,3,\dots,\\
\delta(\omega) &= \sum_{i=1}^N\sum_{k=1}^{K_i}\lambda_{ik}(1-\beta_{ik}(\omega)).
\end{align*}

\subsection{Marginal queue lengths and waiting times}

The expression for $E(\ee^{-\omega W_{ik}})$ can be obtained by the method used in Section \ref{gatedmarginalw}:
\begin{align}
E\left[\ee^{-\omega W_{ik}}\right] =\,&\prod_{j=1}^{i-1}\sigma_j(\omega)\cdot \big([\omega-\lambda_{ik}(1-\beta_{ik}(\omega))]E(C)\big)^{-1}\nonumber\\
&\cdot
\left[\gamma_1\big(\sum_{j=1}^{i-1}\lambda_{j}(1-\beta_{j}(\omega))
+\lambda_{H}(1-\beta_{H}(\omega))+\lambda_{ik}(1-\beta_{ik}(\omega))\big)\right.\nonumber\\
 &\qquad\left.- \gamma_1\big(\omega+\sum_{j=1}^{i-1}\lambda_{j}(1-\beta_{j}(\omega))+\lambda_{H}(1-\beta_{H}(\omega))\big)\right].\label{lstwikgloballygated}
\end{align}

We can use the distributional form of Little's law to determine the LST of the marginal queue length distribution of $Q_{ik}$. Substituting $\omega := \lambda_{ik}(1-z)$ in \eqref{lstwikgloballygated} yields an expression for $E\left[z^{L_{ik}}\right]$, the GF of the distribution of the number of type $ik$ customers at an arbitrary epoch.


\subsection{Moments}

We can obtain $E(W_{ik})$ by differentiation of \eqref{lstwikgloballygated}:
\begin{align*}
E(W_{ik}) &= \sum_{j=1}^{i-1}E(S_j) + (1 + 2\sum_{j=1}^{i-1} \rho_j + 2 \rho_H + \rho_{ik}) E(C_{1,\textit{res}})\\
&=\sum_{j=1}^{i-1}E(S_j) + (1 + 2\sum_{j=1}^{i-1}\sum_{l=1}^{K_j} \rho_{jl} + 2 \sum_{l=1}^{k-1} \rho_{il} + \rho_{ik}) E(C_{1,\textit{res}}).
\end{align*}
The interpretation of this expression is that the mean waiting time of a type $ik$ customer in a globally gated polling system consists of
the mean residual cycle time, the mean switch-over times $S_1 + \dots + S_{i-1}$, the mean service times of all type $1, \dots, i-1$ customers that arrive during the past cycle time \emph{plus} the residual cycle time, the mean service times of all type $i$ customers with higher priority level than $k$ that arrive during the past cycle time \emph{plus} the mean residual cycle time, and the mean service times of all type $i$ customers of priority level $k$ that have arrived during the past cycle time. Notice that the mean past cycle time is equal to the mean residual cycle time.

\section{Numerical example}

Consider a polling system with two queues, and assume exponential service times and switch-over times. Suppose that $\lambda_1 = \frac{6}{10}, \lambda_2 = \frac{2}{10}, E(B_1) = E(B_2) = 1, E(S_1) = E(S_2) = 1$. The workload of this polling system is $\rho = \frac{8}{10}$. This example is extensively discussed in \cite{winands06} where MVA is used to compute mean waiting times and mean residual cycle times for the gated and exhaustive service disciplines, and in \cite{boonadanboxma2queues2008}, where it is shown that the performance of this system can be improved by dividing jobs into two categories (small and large service times) and giving higher priority to jobs with smaller service times. In this example we illustrate how much further the performance of this system can be improved by increasing the number of priority levels in $Q_1$. Let $K_1$ be the number of priority levels in $Q_1$. We define thresholds $t_1, \dots, t_{K_1-1}$ that divide the jobs arriving in $Q_1$ into $K_1$ classes: jobs with a service time less than $t_1$ receive highest priority, jobs with a service time between $t_1$ and $t_2$ receive second highest priority, and so on. The thresholds $t_1, \dots, t_{K_1-1}$ are determined in such a way that the overall mean waiting time for customers in $Q_1$ is minimised. The values of these thresholds depend on the service discipline and are discussed in \cite{wierman07}. Figures \ref{fig:gated}, \ref{fig:exhaustive}, and \ref{fig:globallygated} show the overall mean waiting time for customers in $Q_1$ versus the number of priority levels $K_1$ for gated, exhaustive and globally gated service respectively. Notice that $K_1 = 1$ corresponds to a system without priorities in which customers are served FCFS. Both plots indicate that most  of the improvement is obtained for values of $K_1$ up to 3 or 4. Creating more than 4 priority levels hardly improves this system anymore. The situation $K_1 \rightarrow \infty$ corresponds to a system with shortest-job-first (SJF) policy, which results in a mean waiting time of 10.38 for gated service, 3.53 for exhaustive service, and 9.75 for globally gated service. The dotted horizontal lines in Figures \ref{fig:gated}, \ref{fig:exhaustive} and \ref{fig:globallygated} indicate these values.

From Figures \ref{fig:gated} -- \ref{fig:globallygated} we can conclude that, as far as overall mean waiting time is concerned, the introduction of priority levels significantly improves the performance of this polling system of two queues. An interesting question is whether the introduction of multiple priority levels still significantly improves the performance of a polling system consisting of more than two queues. In order to answer this question we compare polling systems with identical total occupation rate and identical total mean switch-over time. We only vary the number of queues $N$ and, in order to keep the total occupation rate constant, the arrival intensities $\lambda_i, i = 1,\dots,N$. All systems are symmetrical, which means that in each system $\lambda_i = \Lambda/N, \rho_i = \rho/N, E(B_i) = \beta, E(S_i)=E(S)/N, i=1,\dots,N$. For this example we use $\Lambda=\frac{8}{10}$, $B_i$ exponentially distributed with mean $\beta=1$, and deterministic switch-over times with $E(S)=2$. Figures \ref{fig:gated2} (gated) and \ref{fig:exhaustive2} (exhaustive) show the mean waiting time (which is identical for each queue, because of symmetry) versus the number of queues in the system. In each figure the mean waiting time is plotted for both the FCFS and the SJF policy.
The mean waiting time of SJF provides a lower bound for what can be achieved by introducing priority levels. A conclusion that can be drawn from both figures is that the relative improvement by adding priority levels decreases quickly when the number of queues increases. This suggests that the introduction of priority levels has less influence in a system with more queues, because the server spends a greater part of each cycle serving other queues. Another typical feature that is illustrated in Figures \ref{fig:gated2} and \ref{fig:exhaustive2} is that the mean waiting time of an arbitrary customer \emph{decreases} when the number of queues increases in a \emph{gated} polling system, whereas it \emph{increases} in an \emph{exhaustive} polling system. This can be verified using the pseudo-conservation law. For a polling system without priorities with FCFS service at each queue, Equation \eqref{pseudoconservationlawpriorities} reduces to
\[
\begin{aligned}
\sum_{i=1}^N\rho_{i} E(W_{i}) &=
\frac{\rho}{1-\rho} \sum_{i=1}^{N}\rho_{i}\frac{E(B_{i}^2)}{2E(B_{i})}\\
& + \rho \frac{E(S^2)}{2E(S)} +\left[\rho^2-\sum_{i=1}^N \rho_i^2\right]\frac{E(S)}{2(1-\rho)} + \sum_{i=1}^N E(Z_{ii}).
\end{aligned}
\]
The $Z_{ii}\ (i=1,\dots,N)$ depend on the service discipline that is used and are discussed in Subsection \ref{pseudoconservationlawsection}. For an exhaustive polling system, $E(Z_{ii})=0$, so if we substitute $\rho_i = \rho/N$, we obtain
\[
E(W_{i}) =
\frac{\rho}{1-\rho} \frac{E(B_{i}^2)}{2E(B_{i})} + \frac{E(S^2)}{2E(S)} +\left(1-\frac{1}{N}\right)\frac{\rho E(S)}{2(1-\rho)},
\]
which indeed gets larger as $N$ increases. For gated service, we have $E(Z_{ii})=\frac{\rho_i^2 E(S)}{1-\rho}$, which results in
\[
E(W_{i}) =
\frac{\rho}{1-\rho} \frac{E(B_{i}^2)}{2E(B_{i})} + \frac{E(S^2)}{2E(S)} +\left(1+\frac{1}{N}\right)\frac{\rho E(S)}{2(1-\rho)}.
\]
This expression indeed gets smaller as $N$ increases. In practice this means that mean waiting times can be reduced when gated service is used, by creating more queues and assign arriving customers to a random queue. It is however questionable whether this can be realised without changing the total switch-over time distribution. Notice that for $N \rightarrow \infty$ the mean waiting time in an exhaustive symmetric polling system is equal to the mean waiting time in a gated symmetric polling system.

\newpage
\begin{figure}[h!]
\begin{center}
\includegraphics[width=0.75\linewidth]{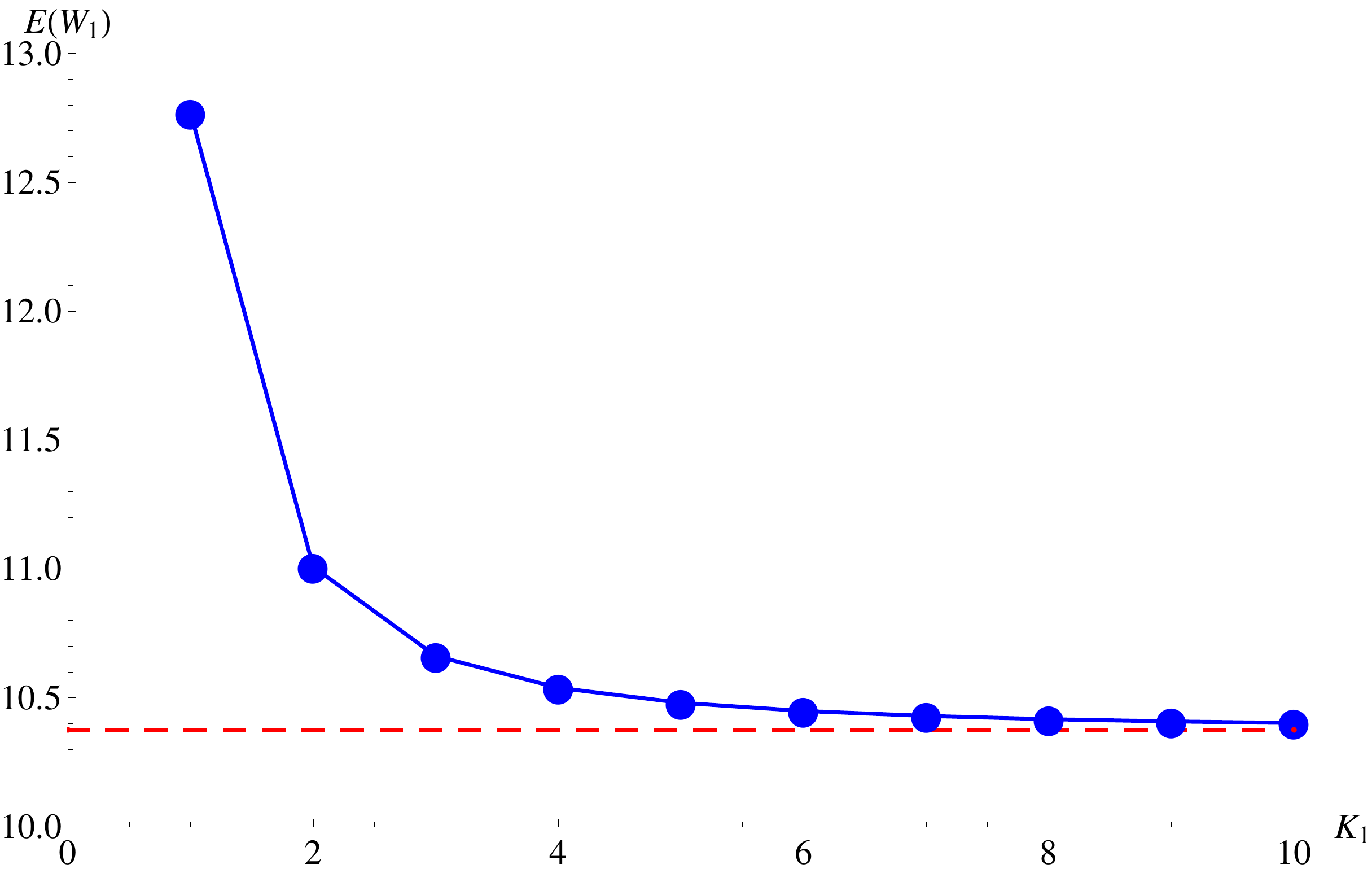}
\end{center}
\caption{Overall mean waiting time of customers in $Q_1$ in the gated polling system, versus number of priority levels $K_1$.\label{fig:gated}}
\end{figure}
\begin{figure}[h!]
\begin{center}
\includegraphics[width=0.75\linewidth]{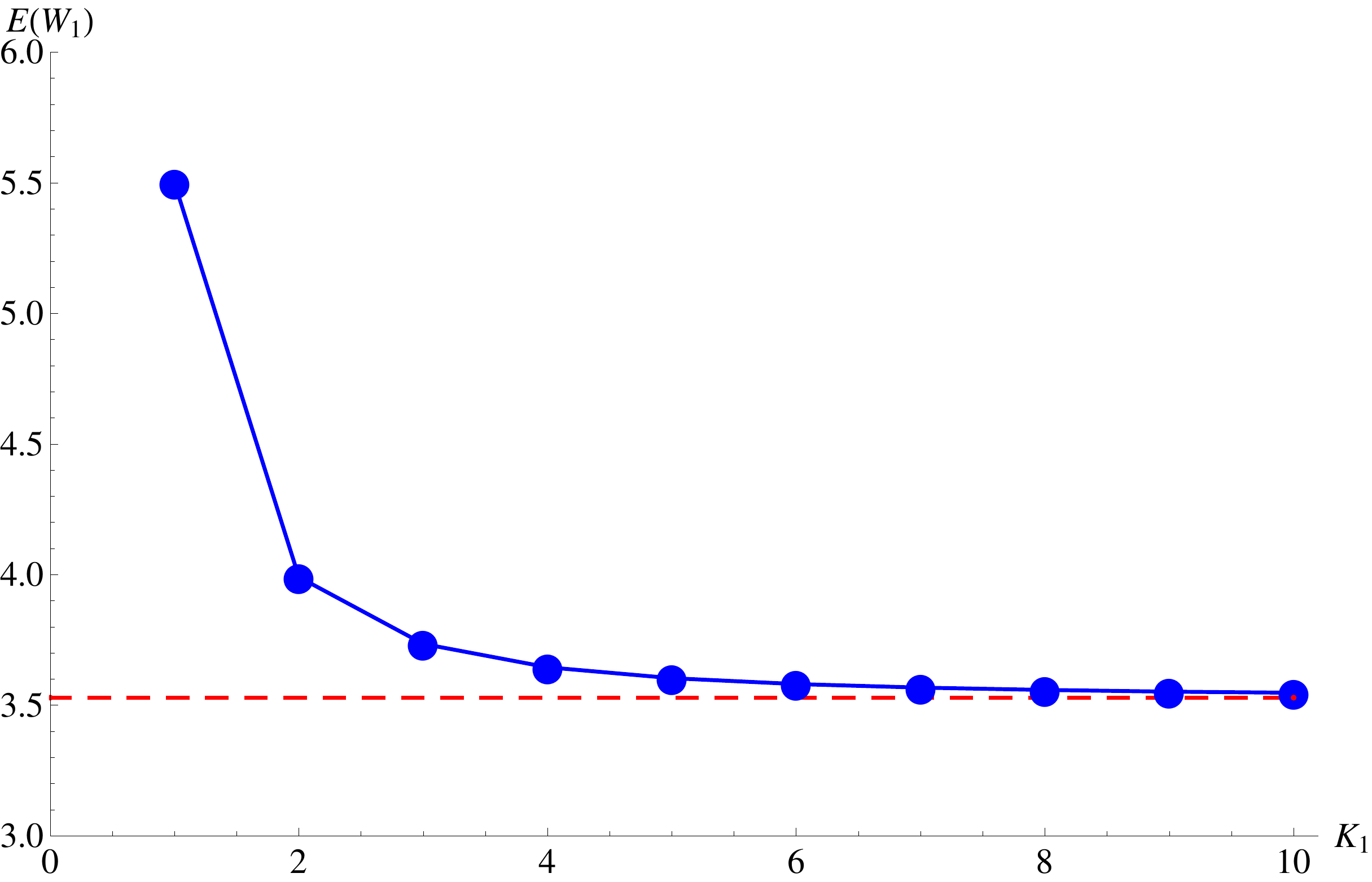}
\end{center}
\caption{Overall mean waiting time of customers in $Q_1$ in the exhaustive polling system, versus number of priority levels $K_1$.\label{fig:exhaustive}}
\end{figure}
\begin{figure}[h!]
\begin{center}
\includegraphics[width=0.75\linewidth]{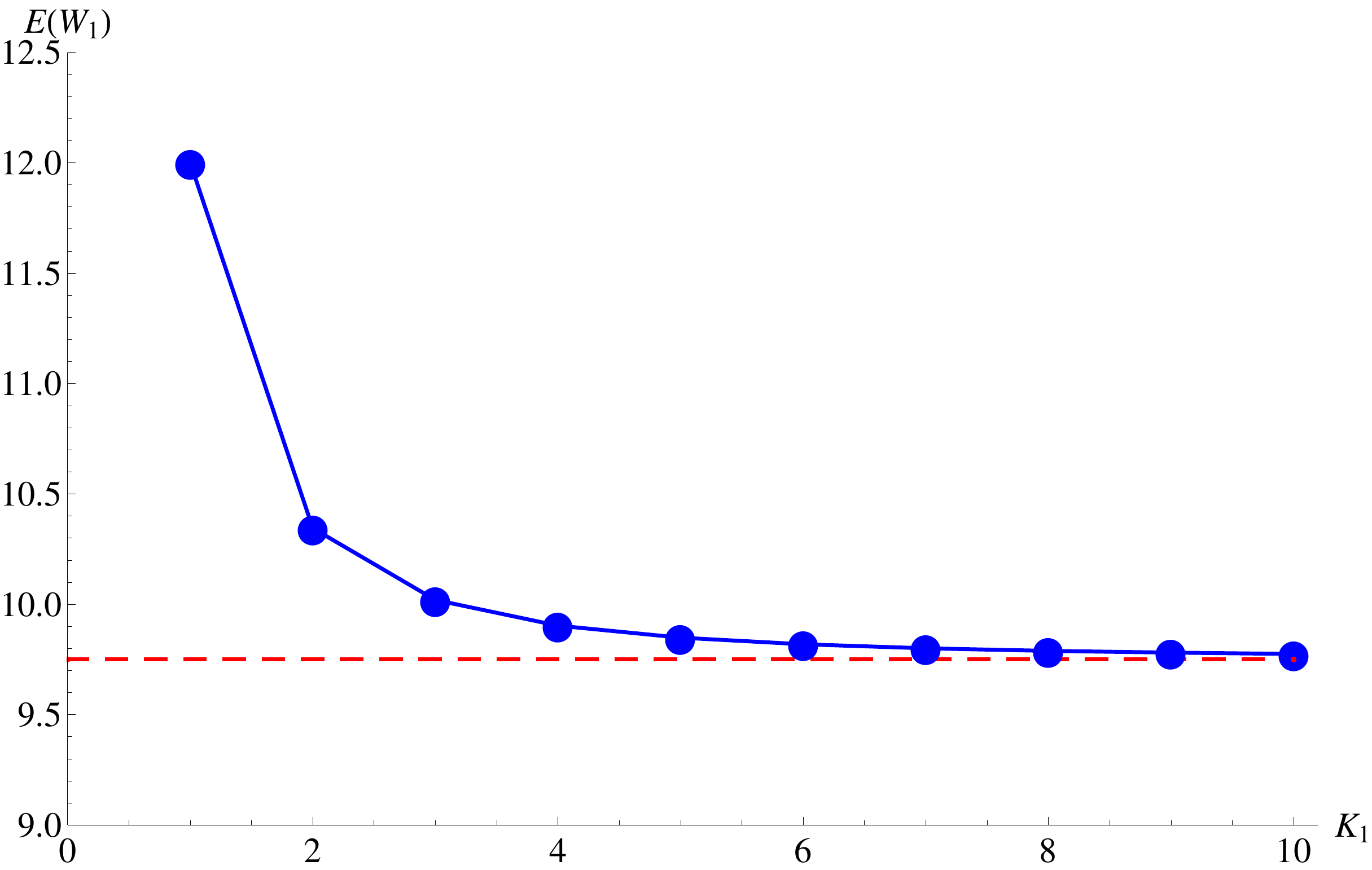}
\end{center}
\caption{Overall mean waiting time of customers in $Q_1$ in the globally gated polling system, versus number of priority levels $K_1$.\label{fig:globallygated}}
\end{figure}
\newpage
\begin{figure}[h!]
\begin{center}
\includegraphics[width=\linewidth]{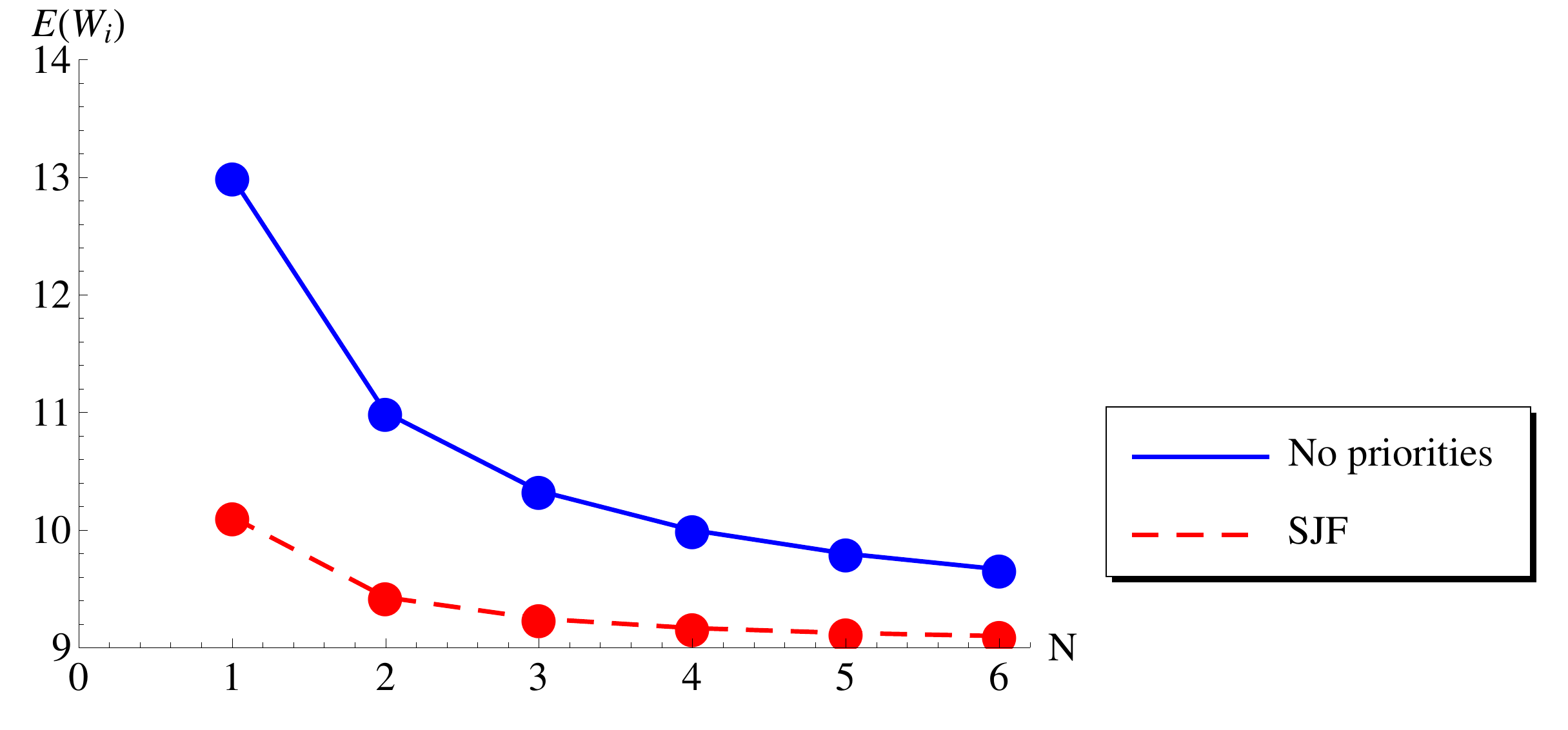}
\end{center}
\caption{Mean waiting time $E(W_i)$ in the gated polling system, versus number of queues.\label{fig:gated2}}
\end{figure}
\begin{figure}[h!]
\begin{center}
\includegraphics[width=\linewidth]{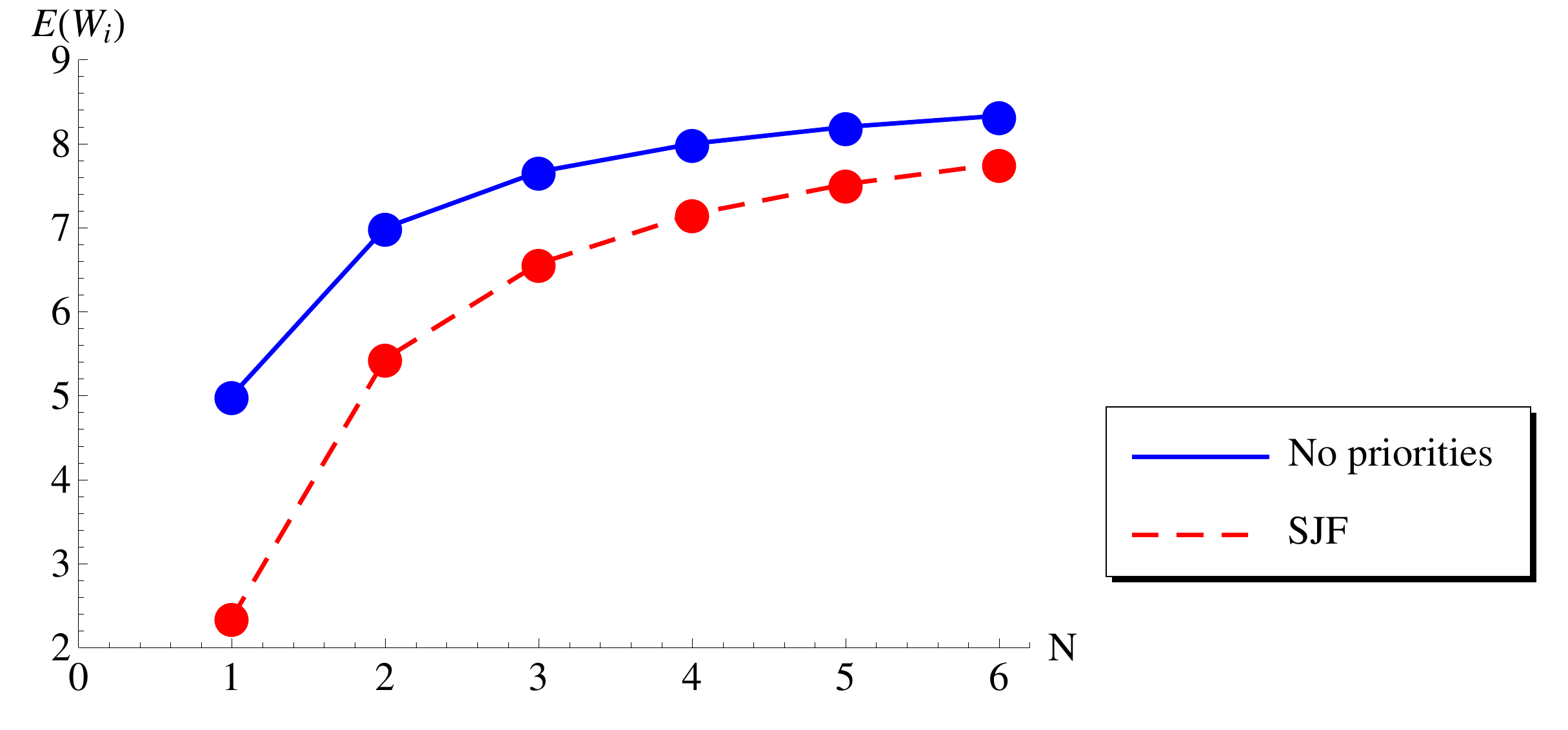}
\end{center}
\caption{Mean waiting time $E(W_i)$ in the exhaustive polling system, versus number of queues.\label{fig:exhaustive2}}
\end{figure}
\newpage
\bibliographystyle{abbrvnat}

\end{document}